\documentclass{agtart_a} 
\pdfoutput=1

\title {Gromov's macroscopic dimension conjecture}

\author{Dmitry V\,Bolotov}
\givenname{Dmitry}
\surname{Bolotov}
\address {B\,Verkin Institute for Low Temperature Physics\\\newline 
Lenina ave 47\\
Kharkov 61103\\Ukraine}
\email {bolotov@univer.kharkov.ua}
\urladdr{}

\volumenumber{6}
\issuenumber{}
\publicationyear{2006}
\papernumber{59}
\startpage{1669}
\endpage{1676}

\doi{}
\MR{}
\Zbl{}


\received{2 March 2006}
\revised{}
\accepted{1 September 2006}
\published{14 October 2006}
\publishedonline{14 October 2006}
\proposed{}
\seconded{}
\corresponding{}
\editor{}
\version{}

\arxivreference{} %%% not in arXiv

\AtBeginDocument{\let\bar\wbar\let\tilde\wtilde\let\hat\what}
\AtBeginDocument{}
\def\C1{\ifreffull Conjecture \fi}
\newtheorem{Cone}{C1}

\makeautorefname{Cone}{C1}
\newtheorem{Ctwo}{C2}

\makeautorefname{Ctwo}{C2}
\def\tildeM{\wwtilde M}
\def\tildeg{\tilde g}
\def\tildef{\tilde f}
\makeop{mc}
\makeop{Sq}

\makeatletter
\def\cnewtheorem#1[#2]#3{\newtheorem{#1}{#3}[section]
\expandafter\let\csname c@#1\endcsname\c@thm}
\makeatother

 \newtheorem{thm}{Theorem}[section]
  \cnewtheorem{cor}[thm]{Corollary}
 \cnewtheorem{lem}[thm]{Lemma}

 \theoremstyle{definition}
 \cnewtheorem{defn}[thm]{Definition}
 \makeautorefname{defn}{Definition}
 \theoremstyle{remark}
 \cnewtheorem{rem}[thm]{Remark}
 
 \newcommand{\eps}{\varepsilon}
 
 \newcommand{\Wi}{\widetilde}

\newcommand{\Diam}{ \text {Diam}}

\begin{document}

\begin{abstract}
In this note we construct a closed 4--manifold having
torsion-free fundamental group and whose universal covering is of
macroscopic dimension 3. This yields a counterexample to
Gromov's conjecture about the falling of macroscopic dimension.
\end{abstract}

\maketitle

\section{Introduction}
The following definition was given by  M\,Gromov \cite{G1}:

\begin{defn}
Let $V$ be a metric space. We say that $\dim_{\varepsilon} V \leq
k$ if there is a $k$--dimensional polyhedron $P$ and a proper
uniformly cobounded map $\phi\co  V \to P$  such that $\text
{Diam}({\ \phi^{-1}(p)})\leq \varepsilon$ for all $p \in P$. A
metric space $V$ has macroscopic $\dim_{\mc} V \leq k$ if
$\dim_{\varepsilon} V \leq k$ for some possibly large $\varepsilon
< \infty$. If $k$ is minimal, we say that
$\dim_{\mc} V = k$.
\end{defn}

Gromov also stated the following questions which, for 
convenience, we state in the form of conjectures:

\begin{Cone} \label{Cone} Let $(M^n,g)$ be a closed Riemannian $n$--manifold
with torsion-free fundamental group, and let $(\tildeM^n,\tildeg)$ be the universal covering of $M^n$ with the
pullback metric. Suppose that $\dim_{\mc} (\tildeM^n,\tildeg) <n$. Then $\dim_{\mc} (\tildeM^n,\tildeg) <n-1$.
\end{Cone}

In \cite{B} we proved \fullref{Cone} for the case $n=3$.

Evidently, the following conjecture would imply \fullref{Cone}  (see also \eqref{itemC} of \fullref{sec2}):

\begin{Ctwo}Let $M^n$ be a closed $n$--manifold with torsion-free fundamental group $\pi$ and let $f\co M^n\to B\pi$ be a classifying
map to the classifying space $B\pi$. Suppose that $f$ is homotopic
to a mapping into the $(n-1)$--skeleton of $B\pi$. Then $f$ is in fact
homotopic to a mapping into the $(n-2)$--skeleton of $B\pi$.
\end{Ctwo}

In this note we show that both conjectures fail for $n\geq 4$.

We always assume that universal covering are equipped with the
pullback metrics.

\subsection*{Framed cobordism, Pontryagin manifolds and classification of mappings to the sphere}
%% ----------------------------------------------------------------------

%In this section we recall som, see~\cite{M} for details.
Let $M$ be a smooth compact manifold possibly with a boundary and let $(N,{ v})$ and $(N',{ w})$ be closed $n$--submanifolds in the interior of $M$ with trivial normal bundles and framings ${ v}$ and ${ w}$, respectively.
%Let us suppose $\partial N =\partial N'=\emptyset$.

\begin{defn}
Two framed submanifolds  $(N,{ v})$ and $(N',{ w})$ are {\it framed cobordant\/} if there exists a cobordism $X\subset M \times[0,1]$ between $N$ and $N{'}$ and a framing $u$ of $X$ such that
$$\begin{array}{lll}
u(x,t)&=(v(x),0)\quad&\text{for }(x,t) \in N\times [0,\eps),\cr
u(x,t)&=(w(x),1)\quad&\text{for }(x,t) \in N'\times (1-\eps,1].
\end{array}$$
\end{defn}

\begin{rem}
If $(N',{ w})=\emptyset$ we  say  $(N,{ v})$ is {\it framed cobordant to zero\/}.
\end{rem}

Now let $f\co M\to S^p$ be a smooth mapping and $y\in S^p$ be a regular value of $f$.
Then $f$ induces the following framing of the submanifold $f^{-1}(y) \subset M$.
Choose a positively oriented basis $v=(v^1\dots,v^p)$ for the tangent space $T(S^p)_y$.
Notice that for each $x \in f^{-1}(y)$ the differential $df_x\co TM_x\to T(S^p)_y$ vanishes on the subspace $Tf^{-1}(y)_x$ and isomorphically maps its orthogonal complement $Tf^{-1}(y)_x^{\perp}$ onto $T(S^p)_y$.
Hence there exists a unique vector $$w^i\in Tf^{-1}(y)_x^{\perp} \subset TM_x$$ which is mapped by $df_x$ to $v^i$.
So we have an induced framing $w=f^*v$ of $f^{-1}(y)$.

\begin{defn}
This framed manifold $(f^{-1}(y),f^*v)$ will be called the {\it
Pontryagin manifold} associated with $f$.
\end{defn}

%The following statements are true \cite{M}:

\begin{thm}{\rm (Milnor \cite{M})}\qua\label{th:1}
If $y'$ is another regular value of $f$ and $v{'}$ is a positively oriented basis for $T(S^p)_{y'}$, then the framed manifold $(f^{-1}(y'),f^*v{'})$ is framed cobordant to $(f^{-1}(y),f^*v)$.
\end{thm}

\begin{thm}{\rm (Milnor \cite{M})}\qua\label{th:2}
Two mappings from $(M,\partial M)$ to $(S^p,s_0)$ are smooth\-ly
homotopic if and only if the associated Pontryagin manifolds are
framed cobordant.
\end{thm}

%\subsection*{Spenier's exact sequence}

%%% ----------------------------------------------------------------------

\section{The construction of an example}
\label{sec2}

Consider a circle bundle $S^3 \times S^1 \to S^2 \times S^1$
obtained by multiplying the Hopf circle bundle $S^3\to S^2$ by
$S^1$. Take also the trivial circle bundle $T^4= S^1\times T^3\to
T^3$ and produce a connected sum
$$M^4=S^3\times S^1 \#_{S^1} T^4$$
of these circle bundles along small tubes consisting of the
circle fibers equipped with natural trivialization. Clearly

\begin{enumerate}
\renewcommand{\theenumi}{\Alph{enumi}}
\item $M^4$ is the total space of the circle bundle  $$p\co M^4 \to M^3= S^2\times S^1 \# T^3;$$

\item $\pi_1(M^4)=\pi_1(M^3)$. Denote this group by $\pi$;

\item \label{itemC} $B\pi= S^1\vee T^3$ and $\dim_{\mc}M^4\leq 3$.
Indeed, the classifying map $f\co  M^4\to B\pi$ can be lifted to the
proper cobounded (by $\Diam (M^4)$) map $\tildef\co \tildeM^4\to \Wi
{B\pi}$ of the universal coverings;

\item the classifying map $f\co  M^4\to B\pi$ can be defined as
the composition
$$
M^4\ \overset{p}{\longrightarrow} \ S^2\times S^1 \# T^3 \  \overset{f_1}{\longrightarrow} \
S^2\times S^1 \vee T^3 \ \overset{f_2}{\longrightarrow} \ S^1 \vee T^3,
$$
where $f_1$ is a quotient map which maps a separating sphere $S^2$
to a point, and $f_2$ is the mapping which coincides with the
projection onto the generating circle of $S^2 \times S^1$ and is
the identity on $T^3$--component.
\end{enumerate}

Let  $g\co S^1 \vee T^3 \to S^3$ be  a degree one map which maps
$S^1$ to a point. Then the following composition $J = g\circ f_2
\circ f_1\co M^3 \to S^3$ also has degree one.

\begin{thm}\label{th:3}
The mapping $f\co M^4\to B\pi$ is not homotopic into the $2$--skeleton of $B\pi$.
\end{thm}
\begin{proof}
Let $\pi\co E \to M^3$ be a two-dimensional vector bundle associated
with the circle bundle $p\co M^4 \to M^3$. Let $E_0$ denote $E$
without  zero section $s\co M^3 \hookrightarrow E$ and $j\co M^4\hookrightarrow E_0$ be a unit circle subbundle of $E$.

The following diagram is homotopically commutative:
\[
\begin{array}{ccc}
M^4 & \stackrel {j}{\hookrightarrow} &
 E_0\\
 \Big\downarrow \vcenter{%
 \rlap{$\scriptstyle{p}$}}& &
\Big\downarrow\vcenter{%
\rlap{$\scriptstyle{\text{embedding}}$}} \\
M^3 & \stackrel {s}{\hookrightarrow}&
 E
\end{array}
\]
Obviously,  $j$ and $s$ are homotopy equivalences.

Recall that we have the Thom isomorphism (see Milnor and Stasheff \cite{MS})
$$\Phi\co  H^k(M^3;\Lambda)\to H^{k+2}(E,E_0;\Lambda)$$ defined by
$$\Phi (x)=(\pi^*x)\cup u,$$
where $\Lambda$ is a ring with unity, and $u$ denotes the Thom
class.

The Thom class $u$ has the following properties \cite{MS}:

\begin{enumerate}
\renewcommand{\labelenumi}{(\alph{enumi})}
\item \label{eq1}If $e$ is the Euler class of $E$ then  we have the Thom--Wu
formula 
$$\Phi (e)= u\cup u.$$
\item $s^*(u)=e$.
\end{enumerate}
Let
$$M_p \ = \ M^4\times I/(x\times 1 \sim p(x))$$ be the  cylinder of the map $p\co M^4 \to M^3$.
Then we have natural embeddings
$$
i_1\co M^4 \to M^4\times 0 \subset M_p \qquad\text{and}\qquad
i_2\co M^3 \to M^3\times 1 \subset M_p
$$
and a natural retraction $r\co  M_p \to M_3$.
It is easy to see that $M_p$ is just a $D^2$--bundle associated to the circle
bundle $p\co M^4 \to M^3$ and $r|_{M^4}= p $.

Recall that the Thom space $(T(E),\infty)$ is the one point
compactification of $E$. Denote $T(E)$ by $T$. Clearly, $T$
is homeomorphic to the quotient space $M_p/M^4$ and
\begin{equation}H^{*}(T,\infty;\Lambda) \cong H^*(E,E_0;\Lambda)\label{eq2}\end{equation} is a
ring isomorphism (see Milnor and Stasheff \cite{MS} for more details).

If $g \circ f \co  M^4 \to S^3$ is nullhomotopic then we
can extend the map $J\co  M_3\times 1\to S^3$ to a mapping $G\co T\to
S^3$. This means that the composition
$$M^3
\overset{i_2}{\longrightarrow} M_p
\overset{\text{quotient}}{\longrightarrow} T \overset{G}{\longrightarrow}
S^3$$ has degree $1$ and $G^*\co H^3(S^3,s_0;\Lambda)\to
H^3(T,\infty;\Lambda)$ is nontrivial.

Let $a\in H^*(E,E_0;\Lambda) $ denote a class corresponding  to
the class $G^*(\bar s)$ by isomorphism \eqref{eq2}, where $\bar s$ is a
generator of $H^3(S^3,\Lambda)$.

Let us consider the following exact sequence of pair :
$$ H^3(E,E_0;\Lambda)\overset {\xi}{\to} H^3(E;\Lambda)\overset {\psi}\to H^3(E_0;\Lambda)$$
Since $E$ is homotopy equivalent to $M^3$, we have
$H^i(E;\Lambda)=H^i(M^3;\Lambda)$. Clearly   $s^*\xi (a)= J^*
(\bar s)$. (Note that $J^* (\bar s)$ is a generator of
$H^3(M^3;\Lambda)$).

Let us note that $e \text{ mod } 2 $ is equal to the Stiefel--Whitney class
$w_2$ which is  nonzero. Indeed, the restriction  of $E$ onto the
embedded  sphere $i\co S^2\subset M^3$ is the  vector bundle
associated with the Hopf circle bundle, and so  $i^*w_2 \not = 0$.
By the Thom construction above there exists a class $z\in
H^1(M^3;\Z_2)$ such that $\Phi (z)=a$. Thus
$$s^*\xi(a)= z\cup w_2 =\{\text {generator of}\ H^3(M^3;\Z_2)\}.$$
Recall the basic properties of Steenrod squares \cite{SE,MS}:

\begin{enumerate}
\item For each $n,i$ and $Y\subset X$  there exists an additive homomorphism $$ \Sq ^i\co H^n(X,Y;\Z_2)\to H^{n+i}(X,Y;\Z_2).$$

\item If $f\co (X,Y)\to (X',Y')$ is a continuous map of pairs, then $$\Sq ^i\circ f^*=f^*\circ \Sq ^i.$$

\item If $a\in H^n(X,Y;\Z_2)$, then $\Sq ^0(a)=a$, $\Sq ^n(a)=a\cup a$ and $\Sq ^i(a)=0$ for $i>n.$

\item  We have Cartan's formula:
$$\Sq ^k(a\cup b)
=\sum_{i+j=k}\Sq ^i(a)\cup \Sq ^j(b).$$
\item $\Sq ^1=w_1\cup \co  H^{m-1}(M;\Z_2)\to H^{m}(M;\Z_2)$, where $M$ is a closed smooth manifold and  $w_1$ is the
first Stiefel--Whitney class of the tangent bundle $TM$.
This follows from the coincidence of the class $w_1$ with the first Wu class $v_1$ \cite{MS}.
It is well known that $w_1=0$ if $M$ is an orientable manifold.
\end{enumerate}

Let us show that $\Sq ^2(\Phi(z))\not =0$. Using the properties above,
it is easy to see that $\Sq ^1(z)= \Sq ^2(z)=0$. Using  the Thom--Wu
formula \eqref{eq1},  we have
\begin{align*}\Sq ^2(\Phi (z))&=\pi_*z \cup \Sq ^2(u) \\ &= \pi_*z \cup u \cup
u \\ &= \pi_*z \cup \Phi (w_2)=\Phi (z\cup w_2) \not = 0.\end{align*}
 Whence $0=G^*(\Sq ^2 (\bar s))= \Sq ^2(G^*(\bar s))\not
= 0$. This contradiction implies that the composition $g\circ f \co M^4 \to S^3$ is
not homotopic to zero and $f \co M^4 \to B\pi$ can not be deformed
into the $2$--skeleton of $B\pi$.
\end{proof}

\begin{cor}\label{cor:1}
The Pontryagin manifold $(p^{-1}(m),p^*(w))$ is not cobordant to
zero, where $(m,w)$  is any framed point of $M^3$.
\end{cor}
\begin{proof}
Indeed, from \fullref{th:3} and \fullref{th:2} it follows that if $s\in S^3$
is a regular point of $g\circ f\co M^4\to S^3$, then the Pontryagin
manifold $(f^{-1}(g^{-1}(s),f^*(g^*(v))$ is not cobordant to zero,
where $v$ is a framing at $s$.
Thus the Pontryagin manifold $(p^{-1}(m),p^*(w))$ for $(m,w)=(J^{-1}(s),(J^{*}(v))$ is also not cobordant to zero.
Now the statement follows from \fullref{th:1} and regularity of the map $p\co M^4\to M^3$.
\end{proof}

\section{The main theorem}

\begin{defn}
A metric space is called uniformly contractible (UC) if there
exists an increasing   function $Q\co \R_+\to \R_+$ such that each
ball of radius $r$ contracts to a point inside a ball of radius
$Q(r)$.
\end{defn}
It is well known that the universal covering of a compact
$K(\tau,1)$ space is UC (see Gromov \cite{G1} for more details).

 Denote by $\rho$ the distance function on $\Wi {B\pi}$.
\begin{lem}\label{lem:2}
Let $\tildef\co \tildeM^4\to \Wi {B\pi}$ be a lifting of a classifying
map to the universal coverings.  If $\dim_{\mc}\tildeM^4\leq 2$,
then there exists a short homotopy $\tilde F\co \tildeM^4\times I\to \Wi
{B\pi}$ of $\tildef$ such that $\tilde F(x,0)=\tildef (x)$ and $\tilde
F(x,1)$ is a through mapping $$\tilde F(x,1)\co \tildeM^4 \to P^2 \to \Wi
{B\pi},$$ where $P^2$ is a $2$--dimensional polyhedron and ``short
homotopy'' means that we have $\rho (\tildef(x),\tilde F (x,t))\leq
\mathrm{const}$ for each $x\in \tildeM^4$, $t\in I$.
\end{lem}
\begin{proof} Let $h\co  \tildeM^4 \to P$ be a proper cobounded continuous
map to some $2$--dimen\-sion\-al polyhedron $P$. Using a simplicial
approximation of $h$, we can  suppose that $h$ is a simplicial map
between such triangulations of $\tildeM^4$ and $P$,  that the
preimage of the star of each vertex is uniformly bounded (recall
that $h$ is proper). Since $\tildef$ is a quasi-isometry, the $\tildef$--image $\tildef(h^{-1}(St(v)))$ of the preimage of the star of
each vertex $v\in P$ is bounded by some constant $d$. Let $M_h$ be
the cylinder of $h$ with natural triangulation consisting of the
triangulations of $\tildeM^4$ and $P$ and the triangulations of the simplices $\{v_0,\dots,v_k ,
h(v_k),\dots,h(v_p)\}$, where $\{v_0,\dots,v_p\}$ is a simplex in
$\tildeM^4$ with $v_0<v_1<\dots,<v_p$ \cite{S}.

Consider the map $\tildef_0\co (M_h)^0\to \Wi {B\pi}$ from $0$--skeleton
$(M_h)^0$ of $M_h$ which coincides with $\tildef$ on the lower base
of $(M_h)^0$ and with the composition $\tildef\circ t_0 $ on the
upper base of $(M_h)^0$, where $t_0\co (P)^0\to \tildeM^4$  is a section of 
$h$ defined on  the $0$--skeleton $(P)^0$ of $P$. Since $\Wi {B\pi}$ is
uniformly contractible, we can extend $\tildef_0 $ to  $M_h$ using
the function $Q$ of the definition of UC-spaces as follows:

By the construction above, $\tildef_0$--image of every two
neighbouring vertexes of  $M_h$ lies into a ball of radius $d$.
Therefore we can extend the map $\tildef_0$ to a mapping $\tildef_1\co  (M_h)^1\to \Wi {B\pi}$ such that $\rho (\tildef(x),\tildef_1(x,t))\leq d$, $x\in (\tildeM^4)^0$. The $\tildef_1$--image of the
boundary of arbitrary  2--simplex of $M_h$ lies into a ball of
radius $3d$. So we can extend $\tildef_1$ to a mapping $\tildef_2\co 
(M_h)^2\to \Wi {B\pi}$ so that $\rho (\tildef(x),\tildef_2(x,t))\leq
4Q(3d)$, $x\in (\tildeM^4)^1$. Similarly, continue $\tildef_2$ to mappings
$\tildef_3,\dots, \tildef_5$ defined on skeletons $(M_h)^3,\dots,
(M_h)^5=M_h$ respectively, so that $\rho (\tildef(x),\tildef_5(x,t))\leq c$, where $c$ is a constant.
\end{proof}

\medskip
{\bf Main Theorem}\qua {\sl $\dim_{\mc}\tildeM^4 =3 $.}
\medskip
 \begin{proof} Let $q\co \Wi
{B\pi}\to \Wi {B\pi}/(\Wi {B\pi}\setminus D^3)\cong S^3$ be a
quotient map, where $D^3$ is an embedded open $3$--dimensional
ball.

Suppose that $\dim_{\mc}\tildeM^4\leq 2$ and let $h\co  \tildeM^4 \to P$ be a
proper cobounded continuous map to some $2$--dimensional
polyhedron $P$ as in \fullref{lem:2}. It is not difficult to find a
compact smooth submanifold with boundary $W\subset \tildeM^4$ such
that $W$ contains a ball of arbitrary fixed radius $r$. Since $\tildef$ is a quasi-isometry, using \fullref{lem:2} we can choose
$r$ big enough such that $\bar D^3\subset \tildef (W)$ and $\tilde
F(\partial W\times I)\cap \bar D^3=\emptyset$, where $\tilde F$
denotes the short homotopy from \fullref{lem:2}. Thus we have a
homotopy
$$q\circ \tilde F\co  (W,\partial W)\times I\to (S^3,s_0)$$
which maps $\partial W\times I$ into the base point $s_0$.  Since
$\dim P=2$,  from \fullref{lem:2} it follows that $q\circ \tilde
F(x,1)$ is homotopic to zero. Therefore  $q\circ \tilde F(x,0)=
q\circ \tildef$ is homotopic to zero (and   $q\circ \tildef$ is
smoothly homotopic to zero   \cite{M}). Let $(s,v)$ be a framed
regular point in $S^3$ for the map $q\circ \tildef$. Then the Pontryagin
manifold $$( \tildef^{-1}\circ q^{-1} (s),\tildef^*q^*(v))$$ must be
cobordant to zero (see \fullref{th:2}). Let $(\Wi\Omega, w)$
be a framed nullcobordism which is embedded in $W\times I$
with the boundary $(\tildef^{-1}\circ q^{-1} (s),\tildef^*q^*(v))$.

Consider the covering map $\tau\co  \tildeM^4 \times I \to M^4 \times
I$. Then $\tau ( \tildef^{-1}\circ q^{-1} (s),\tildef^*q^*(v))$ is an
embedded  framed submanifold  of $M^4$ which coincides with the
Pontryagin manifold $(p^{-1}(m),p^*(\nu))$ of some framed point
$(m,\nu)\in M^3$. And $\tau (\Wi \Omega, w)$ is an immersed framed
submanifold of $M^4 \times I$. Using the Whitney Embedding Theorem
\cite{W}, we can make a small perturbation of $\tau (\Wi
\Omega, w)$ identically on the small collar of the boundary to
obtain a framed nullcobordism with the boundary
$\tau ( \tildef^{-1}\circ q^{-1} (s),\tildef^*q^*(v))$. But this is
impossible by \fullref{cor:1}.
\end{proof}

\begin{rem} By similar arguments one can prove that
$$\dim_{\mc} (\Wi {M^4\times T^p}) = p+3.$$
\end{rem}

{\bf Question}\qua Does $M^4\times T^p$ admit a PSC--metric for some $p$?

\subsubsection*{Acknowledgements} I  thank  Professor Gromov for useful discussions and attention to
this work during my visit to the IHES in December 2005. Also I thank
the  referee  for the useful remarks and S\,Maksimenko for the
help in preparation of the article.

\bibliographystyle{gtart}
\bibliography{link}

\end{document}